\documentclass[11pt,reqno]{amsart}
\usepackage{amsthm,amsmath,amssymb}
\usepackage{graphicx}
\usepackage{float}
\usepackage[colorlinks=true,
linkcolor=blue,
urlcolor=blue,
citecolor=blue]{hyperref}
\usepackage[T1]{fontenc}
\usepackage{mathtools}
\usepackage{relsize}
\usepackage{ytableau}
\usepackage{tikz}
\usepackage{enumerate}
\usepackage[toc,page]{appendix}
\usepackage{lscape}
\usepackage{multirow}
\usepackage{adjustbox,expl3,etoolbox} 
\usepackage[margin = 2.8cm]{geometry}
\usepackage{diagbox}
\usepackage{xcolor}
\usepackage{longtable}
\usepackage{comment}
\usepackage{mathpple}
\usepackage{times}

\newtheorem{thm}{Theorem}[section]
\newtheorem{lem}[thm]{Lemma}

\newtheorem{cor}[thm]{Corollary}

\newtheorem{rmk}[thm]{Remark}

\theoremstyle{definition}

\DeclareMathOperator\Aut{Aut}

\DeclareMathOperator\sym{Sym}
\DeclareMathOperator\alt{Alt}
\newcommand{\agl}[2]{\operatorname{AGL}_{#1}(#2)}
\newcommand{\pgl}[2]{\operatorname{PGL}_{#1}(#2)}
\newcommand{\psl}[2]{\operatorname{PSL}_#1(#2)}

\newcommand{\pgammal}[2]{\operatorname{P\Gamma L}_#1(#2)}

\newcommand{\gl}[2]{\operatorname{GL}_{#1}(#2)}

\newcommand{\gammal}[2]{\operatorname{\Gamma L}_{#1}(#2)}

\newcommand{\pg}[2]{\operatorname{PG}_{#1}(#2)}
\newcommand{\dih}[1]{\operatorname{D}_{#1}}

\newcommand{\itbf}[1]{{\bf{{\emph{{#1}}}}}}

\letcs\replicate{prg_replicate:nn}

\begin{document}	
	
	\title[]{On the maximum intersecting sets of the general semilinear group of degree $2$}
	
	\author[R. Maleki]{Roghayeh Maleki\textsuperscript{$a$}}	
	\author[A.S. Razafimahatratra]{Andriaherimanana Sarobidy Razafimahatratra\textsuperscript{$a,b$,*}}

	
	
	\date{\today}
	
	\maketitle
	
	\begin{center}
		\vspace*{-0.5cm}
		{\Small \it \textsuperscript{a}University of Primorska, UP FAMNIT, Glagolja\v{s}ka 8, 6000 Koper, Slovenia\\
		\textsuperscript{b}University of Primorska, UP IAM, Muzejski trg 2, 6000 Koper, Slovenia\\
		\textsuperscript{*} Corresponding author:  \href{mailto:sarobidy@phystech.edu}{sarobidy@phystech.edu}}
	\end{center}
	
	\begin{abstract}
		Let $p$ be a prime and $q = p^k$. A subset $\mathcal{F} \subset \gammal{2}{q}$ is intersecting if any two semilinear transformations in $\mathcal{F}$ agree on some non-zero vector in $\mathbb{F}_q^2$. We show that any intersecting set of $\gammal{2}{q}$ is of size at most that of a stabilizer of a non-zero vector, and we characterize the intersecting sets of this size. Our proof relies on finding a subgraph which is a lexicographic product in the derangement graph of $\gammal{2}{q}$ in its action on non-zero vectors of $\mathbb{F}_q^2$. This method is also applied to give a new proof that the only maximal intersecting sets of $\gl{2}{q}$ are the maximum intersecting sets.
	\end{abstract}
	
	\section{Introduction}
	
	The Erd\H{o}s-Ko-Rado (EKR) theorem \cite{erdos1961intersection} is an extremal set theory result that characterizes the largest collection of $k$-subsets of $[n]:= \{1,2,\ldots,n \}$ with the property that any two $k$-subsets are not disjoint. We state the EKR theorem as follows.
	\begin{thm}[EKR]
		Let $k\leq n$ be two positive integers such that $2k\leq n$. If $\mathcal{F}$ is a collection of $k$-subsets of $[n]$ such that $A\cap B \neq \varnothing$ for all $A,B\in \mathcal{F}$, then 
		\begin{align}
			| \mathcal{F}| \leq \binom{n-1}{k-1}.\label{eq1}
		\end{align}
		If $n\geq 2k+1$, then equality holds in \eqref{eq1} if and only if $\mathcal{F}$ is a \itbf{canonical intersecting set}, that is, there exists $a\in [n]$ such that
		\begin{align*}
			\bigcap_{A\in \mathcal{F}} A = \{a\}.
		\end{align*}
	\end{thm}
	
	We are interested in extending the EKR theorem to other combinatorial objects (see \cite{godsil2016erdos} for examples of these). In this paper we will focus on an extension of the EKR theorem for a transitive action of the general semilinear group of degree $2$. 
	
	All groups considered in this paper are finite, and all group actions are left actions. Given a finite transitive group $G\leq \sym(\Omega)$, we say that $\mathcal{F} \subset G$ is \itbf{intersecting} if for any $g,h\in G$, there exists $\omega \in \Omega $ such that $g(\omega) = h(\omega)$. Note that the stabilizer $G_\omega= \{ g\in G: g(\omega) = \omega \}$ of $\omega$ in G and its cosets are intersecting sets of $G$. We say that $G\leq \sym(\Omega)$ has the \itbf{EKR property} if any intersecting set $\mathcal{F} \subset G$ has size at most $|G_\omega| = \frac{|G|}{|\Omega|}$. If the largest intersecting sets of $G\leq \sym(\Omega)$ are cosets of a stabilizer of an element of $\Omega$, then we say that $G$ has the \itbf{strict-EKR property}. The study of permutation groups having these properties is currently a very active research area. For instance, all doubly transitive or $2$-transitive groups have been proved to have the EKR property \cite{meagher2016erdHos}. Some $2$-transitive groups such as the ones corresponding to the natural actions of $\sym(n)$ and $\alt(n)$ on $[n]$, for $n\geq 5$, have the strict EKR property, however there are also $2$-transitive groups that do not have the strict-EKR property such as $\pgl{3}{q}$ in its natural action on the points of the projective plane $\pg{2}{q}$. In general, Spiga showed in 2019 that $\pgl{n}{q}$ in its action on the points of the projective space $\pg{n-1}{q}$ does not have the strict-EKR property \cite{spiga2019erdHos} (but has the EKR property). When the degree of this linear group is restricted to $2$, the situation is much different, as proved in 2011 by Meagher and  Spiga \cite{meagher2011erdHos}.
	\begin{thm}[Meagher-Spiga]
		If $q$ is a prime power, then the group $\pgl{2}{q}$ acting on the projective line $\pg{1}{q}$ has the EKR and strict-EKR properties.\label{thm0}
	\end{thm}
	
	The group $\pgl{2}{q}$ acting on $\pg{1}{q}$ is of course $2$-transitive, and the techniques used in \cite{meagher2011erdHos} very much relies on this fact. When the transitive group in question is no longer $2$-transitive, then some of the techniques used in \cite{meagher2011erdHos} do not work anymore, and so other techniques are required. {The work in this paper will explore one such techniques.}

	If $\mathcal{F} \subset G$ is intersecting, then we may assume without loss of generality that the identity element of $G$ is contained in $\mathcal{F}$, otherwise, we can just replace $\mathcal{F}$ with $g^{-1}\mathcal{F}$, for some $g\in \mathcal{F}$. Intersecting sets containing the identity are known as \itbf{basic intersecting sets}. Note that every element of a basic intersecting set must fix an element. Henceforth, we always assume that an intersecting set is a basic intersecting set. Further, we may also define the relation $\equiv_G$ on maximum intersecting sets (with the previous assumption) of $G$ by $\mathcal{F} \equiv_G \mathcal{F}^\prime$ if and only if there exists $g\in G$ such that 
	\begin{align*}
		\mathcal{F}^\prime = g\mathcal{F}g^{-1},
	\end{align*}
	for any intersecting sets $\mathcal{F},\mathcal{F}^\prime$ of maximum size. The relation $\equiv_G$ is an equivalence relation. Given $\omega\in \Omega$, by transitivity of $G\leq \sym(\Omega)$, we can always find a representative of an equivalence class of $\equiv_G$ containing an element fixing $\omega$. By abuse of terminology, we may always fix $\omega\in \Omega$, and refer to a system of distinct representatives of equivalence classes of $\equiv_G$, such that each representative contains an element fixing $\omega$, as the maximum intersecting sets of $G$. 
	For instance, in the case of $\pgl{2}{q}$ in its action on $\pg{1}{q}$, there is exactly one equivalence class of $\equiv_{\pgl{2}{q}}$, the maximum intersecting set is the stabilizer of the element $\langle e_1 \rangle\in \pg{1}{q}$.

	The \itbf{derangement graph} $\Gamma_G$ of a transitive group $G\leq \sym(\Omega)$ is the graph whose vertex set is $G$ and two elements $g,h\in G$ are adjacent in $\Gamma_G$ if and only if $h^{-1}g$ is a \itbf{derangement} or a \itbf{fixed-point-free} permutation of $G$. This is clearly the Cayley graph of $G$ with connection set equal to $\operatorname{Der}(G)$; the set of all derangements of $G$. In particular, $\Gamma_G$ is a vertex-transitive graph. The importance of the derangement graph comes from the fact that an intersecting set of $G$ is a \itbf{coclique} of $\Gamma_G$, and vice versa. Therefore, $G$ has the EKR property if and only if $\alpha(\Gamma_G) = |G_\omega|$, and $G$ has the strict-EKR property if any coclique of maximum size is a coset of a stabilizer of a point of $G.$
	
	We denote the field with $q$ elements by $\mathbb{F}_q$, and its group of unity by $\mathbb{F}_q^*$. For any prime power $q$, the general linear group $\gl{n}{q}$ acts faithfully and transitively, but not $2$-transitively\footnote{a pair of parallel vectors cannot be mapped by an invertible matrix to a pair of non-parallel vectors}, on the non-zero vectors of $\mathbb{F}_q^n$. A \itbf{Singer cycle} of $\gl{n}{q}$ is an element of order $q^n-1$, and a \itbf{Singer subgroup} is a regular group generated by a Singer cycle. A Singer subgroup of $\gl{n}{q}$ always exists, and one such elements can be obtained by embedding the group of unity of $\mathbb{F}_{q^n}$ into $\gl{n}{q}$. A Singer subgroup is a clique of size $q^n-1$ in the derangement graph of $\gl{n}{q}$, and by the \itbf{clique-coclique bound} \cite{godsil2016erdos}, one can deduce that a coclique of  $\gl{n}{q}$ has size at most $\frac{|\gl{n}{q}|}{q^n-1}$, and so $\gl{n}{q}$ has the EKR property. If we let $e_i$, for $1\leq i\leq n$, be the canonical vector of $\mathbb{F}_q^n$, then the stabilizer of $e_1$ in $\gl{n}{q}$ is an intersecting set of maximum size. It is not hard to see that there is another intersecting set of maximum size that is not conjugate to the latter. The subgroup $H$ of $\gl{n}{q}$ fixing the (affine) hyperplane $e_1+ \langle e_2,e_2,\ldots,e_n \rangle$ is also intersecting since for any $A,B\in H$, we have $A^{-1}B \left(e_1 + \langle e_2,\ldots,e_n \rangle\right) = e_1 + \langle e_2,\ldots,e_n \rangle$, which implies that $A^{-1}B \langle e_2,\ldots,e_n\rangle = \langle e_2,\ldots,e_n \rangle$ and $A^{-1}Be_1 - e_1 = (A^{-1}B-I_n)e_1 \in \langle e_2,\ldots,e_n \rangle$. Hence, $\operatorname{rank}(A^{-1}B-I_n) \leq n-1$, so $(A^{-1}B-I_n)x = 0$ always has a non-trivial solution in $x\in \mathbb{F}_q^n$. Therefore, $H$ is intersecting and $|H| = \frac{|\gl{n}{q}|}{q^n-1}$. Since these subgroups are not conjugate in $\gl{n}{q}$, $\gl{n}{q}$ does not have the strict-EKR property. In addition, there are exactly two non-conjugate subgroups of $\gl{n}{q}$, both isomorphic to $\agl{n-1}{q}$, that are intersecting of maximum size. It was recently proved in \cite{forbidden} that for $n$ large enough, an intersecting set of maximum size in the group $\gl{n}{q}$ in its action on non-zero vectors of $\mathbb{F}_q^n$ is one of these two non-conjugate subgroups. In \cite{ernst2023intersection}, Ernst and Schmidt also proved other EKR type results on $\gl{n}{q}$, and partially obtained some of the aforementioned results.
	
	Since these results characterizing the largest intersecting sets in $\gl{n}{q}$ only work for $n$ large enough, characterizing the largest intersecting sets for every $n\geq 3$ is still an open problem. When $n = 2$, it was shown in \cite{ahanjideh2022largest} that the largest intersecting sets in $\gl{2}{q}$ are the stabilizer of the canonical vector $e_1$ or the stabilizer of a hyperplane $e_2+\langle e_1\rangle$. Meagher and the second author also proved some result on the characteristic vectors of the largest intersecting sets of $\gl{2}{q}$ in \cite{meagher2022some}.
	
	In this paper, we prove an EKR type theorem for the \itbf{general semilinear group} $\gammal{2}{q}$ in its action on the non-zero vectors of $\mathbb{F}_q^2$ by \itbf{semilinear transformations}. We show that in contrast to $\gl{2}{q}$, the largest intersecting sets of $\gammal{2}{q}$ are much more complex.	
	Let $p$ be a prime number and $q = p^k$ for some integer $k\geq 1$. Recall that the \itbf{Frobenius automorphism} $\varphi \in \Aut(\mathbb{F}_{q}/\mathbb{F}_p) = \operatorname{Gal}(\mathbb{F}_q/\mathbb{F}_p)$ is the field automorphism such that $\varphi(x) = x^p$, for any $x\in \mathbb{F}_{q}$. The general semilinear group of degree $2$ is the group 
	\begin{align}
		 \gammal{2}{q} = \gl{2}{q} \rtimes \langle \varphi \rangle = \left\{ A\varphi^i: A\in \gl{2}{q} \mbox{ and } 0\leq i\leq k-1 \right\}
	\end{align}
	where for any $A\in \gl{2}{q}$, $\varphi A\varphi^{-1} = A^\varphi$ is the matrix obtained by applying $\varphi$ entrywise. The group $\gammal{2}{q}$ acts naturally on the set of non-zero vectors of $\mathbb{F}_q^2$ by defining, for any $A\varphi^i \in \gammal{2}{q}$, that
	\begin{align*}
		(A\varphi^i) v = Av^{\varphi^i},
	\end{align*}
	where $v^{\varphi^{i}}$ is the vector of $\mathbb{F}^2_q$ obtained by applying $\varphi^i$ entrywise.
	
	Let $\omega \in \mathbb{F}_q$ be a primitive element, that is, $\mathbb{F}_q^* = \langle \omega\rangle$. For any $0\leq i\leq q-2$ and $0\leq j\leq k-1$, we define 
	\begin{align*}
		V_{i,j} := \left\{ \omega^i \varphi^j \begin{bmatrix}
			1 & z\\
			0 & y
		\end{bmatrix}: z\in \mathbb{F}_q, y \in \mathbb{F}_q^* \right\}.
	\end{align*}
	For any $0\leq t\leq q-2$ and $0\leq j\leq k-1$, define
	\begin{align*}
		H_{t,j} := 
		\left\{
		\omega^i\varphi^j
		\begin{bmatrix}
			1& z\\
			0& \omega^{t+ip^{k-j}}
		\end{bmatrix}
		:
		0\leq i\leq q-2, z\in \mathbb{F}_q
		\right\}
		=\left\{
		\varphi^j
		\begin{bmatrix}
			x& z\\
			0& \omega^{t}
		\end{bmatrix}
		:
		x\in \mathbb{F}_q^*, z \in \mathbb{F}_q
		\right\}.
	\end{align*}

	The first main result of the paper is the characterization of the largest intersecting sets in odd characteristic.
	\begin{thm}
		Let $p$ be an odd prime and $q = p^k$, for some integer $k\geq 1$. The group $\gammal{2}{q}$ acting on non-zero vectors of $\mathbb{F}_q^2$ has the EKR property, and in particular any intersecting set has size at most $kq(q-1)$. In addition, the intersecting sets of maximum size in $\gammal{2}{q}$ are one of the following types.
		\begin{enumerate}[1)]
			\item Vertical types: 
			\begin{align*}
				V_{0,0} \cup V_{r_1(p-1),1} \cup \ldots \cup V_{r_j(p^j-1),j} \cup \ldots \cup V_{r_{k-1}(p^{k-1}-1),k-1},
			\end{align*}
			such that 
			\begin{align*}
				\begin{cases}
					r_0 = 0& \\
					0\leq r_j\leq \frac{q-1}{p^{\gcd(j,k)}-1}-1 &\mbox{ for }\ 0\leq j\leq k-1\\
					r_j(p^j-1) \equiv r_{j^\prime}(p^{j^\prime}-1) \pmod{p^{j-j^\prime}-1}&\mbox{ for any $0\leq j,j^\prime \leq k-1$.}
				\end{cases}
			\end{align*}
			\item Horizontal types:
			\begin{align*}
				H_{0,0} \cup H_{t_1(p-1),1} \cup \ldots \cup H_{t_j(p^j-1),i} \cup \ldots \cup H_{t_{k-1}(p^{k-1}-1),k-1},
			\end{align*}
			such that 
			\begin{align*}
				\begin{cases}
					t_0 = 0 &\\
					0\leq t_j\leq \frac{q-1}{p^{\gcd(j,k)}-1}-1 &\mbox{ for }\ 0\leq j\leq k-1\\
					t_j(p^j-1) \equiv t_{j^\prime}(p^{j^\prime}-1) \pmod{p^{j-j^\prime}-1}&\mbox{ for any $0\leq j,j^\prime \leq k-1$.}
				\end{cases}
			\end{align*}
		\end{enumerate}
		\label{thm1}
	\end{thm}
	
	{From this result, we can see that an intersecting set of maximum size is not necessarily a subgroup. For instance if $q = 3^3$, by Theorem~\ref{thm1} the set $V_{0,0} \cup V_{2,1}\cup V_{16,2}$ is intersecting of maximum size, but not a subgroup. In contrast, we can show that if $q = p^2$, then the largest intersecting sets are always subgroups.
	\begin{cor}
		If $p$ is an odd prime, then an intersecting set of maximum size of $\gammal{2}{p^2}$ is the stabilizer of $e_1$ or the stabilizer of the hyperplane $e_2+\langle e_1\rangle$. \label{cor:two-layers}
	\end{cor}}
	
	{In even characteristic, the largest intersecting sets of $\gammal{2}{q}$ become much more complex, but we characterize them nevertheless. We state these results in Theorem~\ref{thm3} and Theorem~\ref{thm4} in Section~\ref{sect:cocliques-even}.
	}
	
	A \itbf{maximum intersecting set }is an intersecting set of largest possible size. An intersecting set is called \itbf{maximal} if it is not properly contained in another intersecting set. A \itbf{Hilton-Milner type set} in $G\leq \sym(\Omega)$ is a maximal intersecting set that is not a maximum intersecting set. In \cite{ahanjideh2022largest} and \cite{maleki2021no}, it was proved that there are no Hilton-Milner type intersecting sets in $\gl{2}{q}$. Therefore, any intersecting sets of $\gl{2}{q}$ is contained in a maximum intersecting set, and in particular, a maximal intersecting set must be a maximum one. The proofs are given in \cite{ahanjideh2022largest} and \cite{maleki2021no}, however, we give a proof that is much simpler. 
	\begin{thm}
		Any maximal intersecting set of $\gl{2}{q}$ is maximum. In particular, there are no Hilton-Milner type sets in $\gl{2}{q}$.\label{thm2}
	\end{thm}

	The method that we use to prove all the aforementioned results rely on finding a subgraph which is a lexicographic product in the derangement graph, similar to what was used in \cite{behajaina2022intersection,maleki2023erd}. We also discuss the potential extension of this method for other groups.

	This paper is organized as follows. In Section~\ref{sect:prelim}, we establish some notations and background results. In Section~\ref{sect:nice-graphs}, we determine some lexicographic product in the derangement graph of $\gammal{2}{q}$ and we use these in Section~\ref{sect:cocliques-odd} and Section~\ref{sect:cocliques-even} to determine the maximum cocliques. In Section~\ref{sect:HM}, we give a proof of Theorem~\ref{thm3} and Theorem~\ref{thm4} for the characterization in even characteristic.
		
	\section{Preliminary reduction}\label{sect:prelim}
	Let $Z$ be the center of $\gl{2}{q}$. Recall that $\pgl{2}{q} = \gl{2}{q}/Z$ and $\pgammal{2}{q} = \gammal{2}{q}/Z$. Also, note that $\pgammal{2}{q} = \pgl{2}{q} \rtimes \langle \varphi\rangle$, where $\varphi$ is the Frobenius automorphism of $\mathbb{F}_q$. In addition, both $\pgl{2}{q}$ and $\pgammal{2}{q}$ act $2$-transitively on the projective line $\pg{1}{q}$.
	
	We recall the following lemma.
	\begin{lem}[Theorem~14.7.4 in \cite{godsil2016erdos}]
		Let $G\leq \sym(\Omega)$ be transitive and $H\leq G$ be a $2$-transitive group. If $H$ has the strict EKR property, then so does $G$. 
		\label{lem:2-transitivity}
	\end{lem}
	
	By Theorem~\ref{thm0}, the group $\pgl{2}{q}$ has the EKR and strict-EKR property. Combining this fact with Lemma~\ref{lem:2-transitivity}, we know that $\pgammal{2}{q}$ also has the strict-EKR property. By definition of $\equiv_{\pgl{2}{q}}$, the maximum intersecting sets can be viewed as the elements of a system of distinct representatives of the corresponding equivalence classes, each containing an element fixing $\langle e_1\rangle$. Since $\pgl{2}{q}$ has the strict-EKR property, there is only one equivalence class of $\equiv_{\pgl{2}{q}}$. Let $\mathfrak{F}$ be the maximum intersecting set of $\pgammal{2}{q}$. We can see that $\mathfrak{F}$ is the stabilizer of $\langle e_1\rangle \in \pg{1}{q}$. Therefore,  
	\begin{align}
		\mathfrak{F} = 
		\left\{ \begin{bmatrix}
			1 & z\\
			0 & y
		\end{bmatrix}\varphi^iZ : \ z\in \mathbb{F}_q,\ y\in \mathbb{F}_q^*,\ 0\leq i\leq k-1 \right\} .\label{eq:stab-pgammal}
	\end{align}
	
	Now we turn our focus to the intersecting sets of $\gammal{2}{q}$.
	\begin{lem}
		The group $\gammal{2}{q}$ has the EKR property.\label{lem:EKR}
	\end{lem}
	\begin{proof}
		Let $\Gamma_q$ be the derangement graph of $\gammal{2}{q}$. Note that a Singer cycle in $\gl{2}{q}$ is also a regular subgroup of $\gammal{2}{q}$. Therefore, by the clique-coclique bound, we know that $\alpha(\Gamma_q) \leq \frac{|\gammal{2}{q}|}{q^2-1}$, which is always attained by the stabilizer of a point.
	\end{proof}
	
	If $\mathcal{F} \subset \gammal{2}{q}$ is an intersecting set of size $\frac{|\gammal{2}{q}|}{q^2-1}$, then note that by the canonical surjective homomorphism $\gammal{2}{q} \to \pgammal{2}{q}$ such that $g\mapsto gZ$, we know that $\overline{\mathcal{F}} = \left\{ gZ : g\in \mathcal{F} \right\}$ is also an intersecting set of $\pgammal{2}{q}$ in its action on the $\pg{1}{q}$. In addition, for any $gZ,g^\prime Z \in \overline{\mathcal{F}}$, we have that $gZ = g^\prime Z$ if and only if $g = g^\prime$ (since they are intersecting and $g^{-1}g^\prime \in Z$). The latter implies that $|\overline{\mathcal{F}}| = |\mathcal{F}| = \frac{|\gammal{2}{q}|}{q^2-1} = \frac{|\pgammal{2}{q}|}{q+1}$. Therefore, $\overline{\mathcal{F}} = \mathfrak{F}$, which is defined in \eqref{eq:stab-pgammal}. Now, the preimage of $\mathfrak{F}$ under the canonical surjective homomorphism is the subgroup of $\gammal{2}{q}$ that leaves $\langle e_1\rangle$ invariant. Consequently, the maximum intersecting set $\mathcal{F} \subset \gammal{2}{q}$ is contained in the subgroup
	\begin{align}
		\mathcal{H}_q = \left\{ \begin{bmatrix}
			x & z\\
			0 & y
		\end{bmatrix}\varphi^j:\ z\in \mathbb{F}_q,x,y\in \mathbb{F}_q^*, 0\leq j\leq k-1 \right\} .\label{eq:H}
	\end{align}
	Define
	\begin{align*}
		\mathcal{K}_q = \left\{ \begin{bmatrix}
			1 & z\\
			0 & y
		\end{bmatrix}\varphi^j:\ z\in \mathbb{F}_q,y\in \mathbb{F}_q^*, 0\leq j\leq k-1 \right\} \leq \mathcal{H}_{q}
	\end{align*}
	and 
	\begin{align*}
		\mathcal{M}_q = \left\{ \begin{bmatrix}
			x & z\\
			0 & 1
		\end{bmatrix}\varphi^j:\ z\in \mathbb{F}_q,x\in \mathbb{F}_q^*, 0\leq j\leq k-1 \right\} \leq \mathcal{H}_q.
	\end{align*}
	Note that $\mathcal{K}_q$ is exactly the stabilizer of $e_1$ in $\gammal{2}{q}$ and $\mathcal{M}_q$ is the subgroup leaving the hyperplane $e_2+\langle e_1\rangle$ invariant. 
	\section{Lexicographic product and the derangement graph of $\gammal{2}{q}$}\label{sect:nice-graphs}
	
	Let $\Gamma_q$ be the subgraph of the derangement graph of $\gammal{2}{q}$ induced by $\mathcal{H}_q$. Let $\omega$ be a primitive element of $\mathbb{F}_q$. Define
	\begin{align*}
		\mathcal{N}_q = \left\{ \begin{bmatrix}
			1 & z\\
			0 & y
		\end{bmatrix}:\ z\in \mathbb{F}_q,y\in \mathbb{F}_q^* \right\}\leq \gl{2}{q}.
	\end{align*}
	It is worth mentioning that $\mathcal{H}_q = \left(\mathcal{N}_q\rtimes \langle Z, \varphi \rangle\right) = \mathcal{N}_q \rtimes \left(Z\rtimes \langle \varphi\rangle\right)$ and $\mathcal{K}_q = \mathcal{N}_q \rtimes \langle \varphi \rangle$.
	
	We also define $V_{i,j} =  \omega^i \varphi^j\mathcal{N}_q = \left\{ \omega^i \varphi^jB :\ B\in \mathcal{N}_q\right\}$, for any $0\leq i\leq q-2$ and $0\leq j\leq k-1$. Since $\mathcal{N}_q \leq \mathcal{H}_q$, we know that $\mathcal{H}_q$ is partitioned by $\left(V_{i,j}\right)_{0\leq i\leq q-2,0\leq j\leq k-1}$.
	
	Before proceeding to the main results, we need one last notation. For any $y \in \mathbb{F}_q^*$, we define
	\begin{align}
		\mathcal{E}(y) := \left\{ \begin{bmatrix}
			1 & z\\
			0 & y
		\end{bmatrix}:\ z\in \mathbb{F}_q \right\}\leq \gl{2}{q}.\label{eq:epsilon}
	\end{align}
	Note that $\varepsilon(1)$ is a Sylow $p$-subgroup of $\gl{2}{q}$. It is clear that $\mathcal{N}_q = \bigcup_{q\in \mathbb{F}_q^*} \mathcal{E}(y)$, and so $V_{i,j} = \omega^i \varphi^j \bigcup_{y\in \mathbb{F}_q^*}\mathcal{E}(y)$, for any $0\leq i\leq q-2$ and $0\leq j\leq k-1$. 
	
	For any $n\geq 1$, let $K_n$ be the complete graph on $n$ vertices, and $\overline{K_n}$ its complement. For any integer $n,t\geq 2$, let $K_{n}^t$ be the complete multipartite graph $K_{n,\ldots,n}$ on $nt$ vertices. Consider a subgraph $P$ of $K_{n}^t$ consisting of $n$ vertex-disjoint cliques isomorphic to $K_t$. We define $\widetilde{K}_{n}^t$ to be the graph obtained by removing the edges of $P$ from $K_n^t$. Note that the resulting graphs are all isomorphic, independently on $P$, so we do not specify which unions of cliques to use.
	
	Given two graphs $X$ and $Y$, recall that the \itbf{lexicographic product} of $X$ and $Y$ is the graph $X[Y]$ whose vertex set is $V(X) \times V(Y)$, and two vertices $(x,y)$ and $(x^\prime,y^\prime)$ are adjacent if and only if $x\sim_X x^\prime$ or $x = x^\prime$ and $y\sim_Y y^\prime$.
	\begin{lem}
		For any distinct $0\leq i,i^\prime\leq q-2$ and $0\leq j\leq k-1$, the subgraph of $\Gamma_q$ induced by $\omega^i\varphi^j\mathcal{E}(y) \cup \omega^{i^\prime}\varphi^j\mathcal{E}(y^\prime)$ is a coclique if and only if $y^\prime =y\omega^{(i - i^\prime)p^{k-j}}$. If $y^\prime \neq y\omega^{(i - i^\prime)p^{k-j}}$, then the subgraph is the complete bipartite graph $K_{q,q}$. In particular, the subgraph of $\Gamma_q$ induced by $V_{i,j} \cup V_{i^\prime,j}$ is $\widetilde{K}_{q-1}^2\left[\overline{K_{q}}\right]$.\label{lem:same-layer}
	\end{lem}
	\begin{proof}
		Consider the elements		
		$\omega^i \varphi^j
		\begin{bmatrix}
			1 & z\\
			0 & y
		\end{bmatrix} \in V_{i,j}$ and $\omega^{i^\prime}\varphi^j\begin{bmatrix}
		1 & z^\prime\\
		0 & y^\prime 
		\end{bmatrix} \in V_{i^\prime,j}$. These elements are adjacent if and only if 
		\begin{align*}
			\omega^{(i^\prime -i)p^{k-j}} 
			\begin{bmatrix}
				1 & -zy^{-1}\\
				0 & y^{-1}
			\end{bmatrix}
			\begin{bmatrix}
				1 & z^\prime\\
				0 & y^\prime 
			\end{bmatrix} = \omega^{(i^\prime -i)p^{k-j}} \begin{bmatrix}
			1 & z^\prime -zy^{-1}y^\prime\\
			0 & y^{-1}y^\prime  
			\end{bmatrix}
		\end{align*}
		is a derangement. This happens if and only if $1$ is an eigenvalue of the matrix, that is, if $y^\prime \neq y\omega^{(i-i^\prime)p^{k-j}}$. In addition, the latter is independent of the choice of $z,z^\prime \in \mathbb{F}_q$. If we partition $\mathcal{N}_q$ as cosets of its unique Sylow $p$-subgroup $\mathcal{E}(1)$, then it follows immediately that given one of such cosets in $V_{i,j}$ and one in $V_{i^\prime,j}$, their induced subgraph in $\Gamma_q$ is either a complete bipartite graph or a coclique. This completes the proof.
	\end{proof}
	\begin{cor}
		The subgraph of $\Gamma_q$ induced by $V_{0,j}\cup V_{1,j}\cup \ldots\cup V_{q-2,j}$ is isomorphic to $\widetilde{K}_{q-1}^{q-1}\left[\overline{K_{q}}\right]$, for any $0\leq j\leq k-1$.\label{cor:one-layer}
	\end{cor}

	Now, we determine the adjacency between $V_{i,j}$ and $V_{i^\prime,j^\prime}$, for $0\leq i,i^\prime\leq q-2$ and $0\leq j\leq j^\prime\leq k-1$. Before proceeding to the statements, let us determine when a vertex from $V_{i,j}$ is adjacent to a vertex of $V_{i^\prime,j^\prime}$.
	Consider the arbitrary elements		
	\begin{align}
		\mbox{$\omega^i \varphi^j
			\begin{bmatrix}
				1 & z\\
				0 & y
			\end{bmatrix} \in V_{i,j}$ and $\omega^{i^\prime}\varphi^{j^\prime}\begin{bmatrix}
				1 & z^\prime\\
				0 & y^\prime 
			\end{bmatrix} \in V_{i^\prime,j^\prime}$.}\label{eq:elements-to-check}
	\end{align}
	These elements are adjacent if and only if 
	\begin{align*}
		\left(\omega^i \varphi^j
		\begin{bmatrix}
			1 & z\\
			0 & y
		\end{bmatrix}\right)^{-1}
		\left(\omega^{i^\prime}\varphi^{j^\prime}\begin{bmatrix}
			1 & z^\prime\\
			0 & y^\prime 
		\end{bmatrix}\right)
	\end{align*}
	is a derangement. We have
	\begin{align*}
		\left(\omega^i \varphi^{j}
		\begin{bmatrix}
			1 & z\\
			0 & y
		\end{bmatrix}\right)^{-1}
		\left(\omega^{i^\prime}\varphi^{j^\prime}\begin{bmatrix}
			1 & z^\prime\\
			0 & y^\prime 
		\end{bmatrix}\right)
		&=
		\begin{bmatrix}
			1 & -zy^{-1}\\
			0 & y^{-1}
		\end{bmatrix} 
		\varphi^{-j}\omega^{i^\prime-i}\varphi^{j^\prime}\begin{bmatrix}
			1 & z^\prime\\
			0 & y^\prime 
		\end{bmatrix}\\
		&=
		\begin{bmatrix}
			1 & -zy^{-1}\\
			0 & y^{-1}
		\end{bmatrix} 
		\left(\omega^{i^\prime-i}\right)^{p^{k-j}}\varphi^{j^\prime-j}\begin{bmatrix}
			1 & z^\prime\\
			0 & y^\prime 
		\end{bmatrix}\\
		&=
		\omega^{\left(i^\prime-i\right)p^{k-j}}
		\begin{bmatrix}
			1 & -zy^{-1}\\
			0 & y^{-1}
		\end{bmatrix} 
		\varphi^{j^\prime-j}
		\begin{bmatrix}
			1 & z^\prime\\
			0 & y^\prime 
		\end{bmatrix}\\
		&=
		\omega^{(i^\prime-i)p^{k-j}}
		\varphi^{j^\prime-j}
		\begin{bmatrix}
			1 & \left(-zy^{-1}\right)^{p^{j^\prime-j}}\\
			0 & \left(y^{-1}\right)^{p^{j^\prime -j}}
		\end{bmatrix} 	
		\begin{bmatrix}
			1 & z^\prime\\
			0 & y^\prime 
		\end{bmatrix}\\
		&=
		\omega^{(i^\prime-i)p^{k-j}}
		\varphi^{j^\prime-j}
		\begin{bmatrix}
			1 & u\\
			0 & v
		\end{bmatrix} \in V_{(i^\prime-i)p^{k-j},j^\prime-j},
	\end{align*}
	for some $u\in \mathbb{F}_q$ and $v \in \mathbb{F}_q^*$. Therefore, the adjacencies between $V_{i,j}$ and $V_{i^\prime,j^\prime}$ are equivalent to that of the identity element and the elements of $V_{(i^\prime-i)p^{k-j},j^\prime-j}$. In particular, the vertices in \eqref{eq:elements-to-check} are adjacent if and only if 
	\begin{align}
		\omega^{(i^\prime-i)p^{k-j}}
		\varphi^{j^\prime-j}
		\begin{bmatrix}
			1 & u\\
			0 & v
		\end{bmatrix}\label{eq:element-shifted}
	\end{align}
	is a derangement. If $[a\ b]^t \in \mathbb{F}_q^2$ is fixed by the above element of $\gammal{2}{q}$, then we must have
	\begin{align*}
		\begin{bmatrix}
			a\\
			b
		\end{bmatrix}
		=
		\omega^{(i^\prime-i)p^{k-j}} \varphi^{j^\prime -j}
		\begin{bmatrix}
			a+ub\\
			vb
		\end{bmatrix}
		\Leftrightarrow 
		\varphi^{j-j^\prime}
		\begin{bmatrix}
			a\omega^{(i-i^\prime)p^{k-j}}\\
			b\omega^{(i-i^\prime)p^{k-j}}
		\end{bmatrix}
		=
		\begin{bmatrix}
			a+ub\\
			vb
		\end{bmatrix},
	\end{align*}
	which can be reformulated as
	\begin{align}
		\begin{cases}
			a^{p^{j-j^\prime}} \omega^{(i-i^\prime)p^{k-j^\prime}} &= a+ub\\
			b^{p^{j-j^\prime}} \omega^{(i-i^\prime)p^{k-j^\prime}} &= vb.
		\end{cases}\label{eq:matrix}
	\end{align}
	
	Consequently, the two vertices in \eqref{eq:elements-to-check} are adjacent in $\Gamma_q$ if and only if \eqref{eq:matrix} admits no solutions in $a,b\in \mathbb{F}_q$.
	\begin{lem}
		Assume that $p$ is an odd prime. 
		\begin{enumerate}[(I)]
			\item For any $0\leq i,i^\prime\leq q-2$ and for any $0\leq j< j^\prime \leq k-1$, the subgraph of $\Gamma_q$ induced by $V_{i,j} \cup V_{i^\prime,j^\prime}$ is a coclique if and only if $i^{\prime}\equiv i \pmod {p^{j^\prime-j}-1}$.\label{i}
			\item  If $V_{i,j} \cup V_{i^\prime,j^\prime}$ is not a coclique of $\Gamma_q$, then a vertex of $V_{i,j}$ is adjacent to exactly 
			\begin{align*}
				q^2-q-\left(\frac{q-1}{p^{\gcd\left(k,j^\prime-j\right)}-1}\right)q
			\end{align*}
			vertices in $V_{i^\prime,j^\prime}$.\label{ii}
		\end{enumerate}\label{lem:across}
	\end{lem}
	\begin{proof}		
		\eqref{i} Assume that $i^\prime \equiv i \pmod{p^{j^\prime - j}-1}$ and let us prove that $V_{i,j} \cup V_{i^\prime,j^\prime}$ is a coclique. By the explanation before the statement of the lemma, we only need to show that every element of $V_{(i^\prime-i)p^{k-j},j^\prime-j}$ has a fixed point. In order to do this, we need to show that $\eqref{eq:matrix}$ has a solution in $a$ and $b$, where $(a,b)\neq (0,0)$, independently of $u\in \mathbb{F}_q$ and $v\in \mathbb{F}_q^*$. Let us choose $b = 0$. Then, the second equation of \eqref{eq:matrix} is trivially satisfied.
		The first equation of \eqref{eq:matrix} becomes
		\begin{align*}
			a^{p^{j^\prime - j}-1} = \omega^{(i^\prime-i)p^{k-j}} .
		\end{align*}
		We note that the above equation has a solution in $a\in \mathbb{F}_q$ if and only if $\omega^{-(i^\prime - i)p^{k-j}} \in \langle \omega^{p^{j^\prime-j}-1} \rangle$. Therefore, under the assumption that $b=0$, a solution in $a\in \mathbb{F}_q$ exists if and only if 
		\begin{align}
			\left(p^{j^\prime -j}-1\right) \mid (i^\prime -i) \Leftrightarrow i^\prime \equiv i \pmod {p^{j^\prime -j}-1}.\label{eq:equivalent}
		\end{align}
		 Since $i^\prime \equiv i \pmod{p^{j^\prime - j}-1}$, we conclude that $\omega^{(i^\prime-i)p^{k-j}}
		 \varphi^{j^\prime-j}
		 \begin{bmatrix}
		 	1 & u\\
		 	0 & v
		 \end{bmatrix}$ always fixes a point, and thus not a derangement. Since \eqref{eq:equivalent} is independent of the choice of $u\in \mathbb{F}_q$ and $v\in \mathbb{F}_q^*$, we conclude that $V_{i,j} \cup V_{i^\prime,j^\prime}$ is a coclique in $\Gamma_q.$ 
		 
		Conversely, assume that $V_{i,j}\cup V_{i^\prime,j^\prime}$ is a coclique. Then, for any $u\in \mathbb{F}_q$ and $v\in \mathbb{F}_q^*$,  \eqref{eq:matrix} always has a solution in $a$ and $b$, where $(a,b)\neq (0,0)$. If $b=0$, then a solution in $a$ cannot exist unless $i^\prime\equiv i \pmod{p^{j^\prime-j}-1}$. So we may assume that $b \neq 0$. Similar to what we saw previously, the equation $b^{p^{j^\prime-j}-1} = v^{-p^{j^\prime -j}}\omega^{-(i^\prime-i)p^{k-j}} $ has a solution in $b\in \mathbb{F}_q$ if and only if $v^{-p^{j^\prime -j}}\omega^{-(i^\prime-i)p^{k-j}} \in \langle \omega^{p^{j^\prime -j}-1}\rangle$. Hence, a solution in $b\in \mathbb{F}_q$ exists if and only if 
		\begin{align}
			\mbox{}	v^{-p^{j^\prime -j}} \in \langle \omega^{p^{j^\prime -j}-1}\rangle\omega^{(i^\prime-i)p^{k-j}} \Leftrightarrow v \in \langle \omega^{p^{j^\prime -j}-1}\rangle\omega^{(i^\prime-i)p^{k-j^\prime}}.\label{eq:equivalence}
		\end{align}
		Equivalently, since $v = \omega^t$ for some $0\leq t\leq q-2$, a solution in $b$ exists if and only if 
		\begin{align}
			t\equiv (i^\prime-i)p^{k-j^\prime} \pmod {p^{j^\prime-j}-1}.\label{eq:equivalent-assumption}
		\end{align}
		If $t \equiv 0 \pmod{p^{j^\prime-j}-1}$, then we again have $i^\prime \equiv i \pmod {p^{j^\prime-j}-1}$. Hence, for any $v=\omega^{t_0}$ such that $t_0 \not\equiv (i^\prime-i)p^{k-j^\prime} \pmod {p^{j^\prime-j}-1}$, there is no solution in $b\in \mathbb{F}_q^*$ and the element $\omega^{(i^\prime-i)p^{k-j}}
		\varphi^{j^\prime-j}
		\begin{bmatrix}
			1 & u\\
			0 & \omega^{t_0}
		\end{bmatrix}$ is a derangement \footnote{there always exists such a $t_0$ as long as the characteristic is odd since $\frac{q-1}{p^{\gcd(j^\prime - j,k)}-1}>1$.}, for any $u \in \mathbb{F}_q$. Consequently, $V_{i,j} \cup V_{i^\prime,j^\prime}$ cannot be a coclique, which is a contradiction. The assumption that led to this contradiction is $b\neq 0$, so we must have $b=0$ and $i^\prime \equiv i\pmod{p^{j^\prime-j}-1}$. This completes the proof of \eqref{i}.
		
		\vspace*{0.5cm}
		\eqref{ii} Now assume that $V_{i,j}\cup V_{i^\prime,j^\prime}$ is not a coclique. By \eqref{i}, we know that the latter is equivalent to $i^\prime - i \not\equiv 0 \pmod {p^{j^\prime - j}-1}$.
		In order to prove \eqref{ii}, we need to compute the number of elements in $V_{i^\prime,j^\prime}$ that are non-adjacent to a given vertex of $V_{i,j}$. Again, by the explanation before the statement of the lemma this is the same as the number of non-derangements in $V_{(i^\prime-i)p^{k-j},j^\prime-j}$. The number of adjacent vertices is then obtained by subtracting the number non-derangements from $q^2-q$. 
		
		The element of $V_{(i^\prime-i)p^{k-j},j^\prime-j}$ given in \eqref{eq:element-shifted} is a non-derangement if and only if $\eqref{eq:matrix}$ has a solution in $a$ and $b$. We claim that given $v\in \mathbb{F}^*_q$, if a solution $b\in \mathbb{F}_q$ exists in \eqref{eq:matrix}, then regardless of the choice of $u\in \mathbb{F}_q$, a solution in $a\in \mathbb{F}_q$ from \eqref{eq:matrix} must exist. Assume that a solution in $b\in \mathbb{F}_q$ exists. Then, from the proof of \eqref{i} we must have $t\equiv (i^\prime-i)p^{k-j^\prime} \pmod {p^{j^\prime-j}-1}$, where $0\leq t\leq q-2$ such that $v = \omega^t$. Now consider the second equation of \eqref{eq:matrix}, which is
		\begin{align}
			a^{p^{j-j^\prime}}\omega^{(i-i^\prime)p^{k-j^\prime}}  &= a+ub.\label{eq:main-a} 
		\end{align}
		
		Note that if $b = 0$, then the only way that a solution in $a\in \mathbb{F}_q$ exists is if $i^\prime \equiv i \pmod {p^{j^\prime -j}-1}$, which is not the case. Hence, $b\neq 0$, and we conclude that  whenever $u = 0$, we can choose $a = 0$.
		
		Suppose that $u\neq 0$. If $a$ exists, then $a\neq 0$ otherwise $ub = 0$ but $u\neq 0$ and $b\neq 0$. Therefore, we may reformulate \eqref{eq:main-a} as
		\begin{align*}
			a^{p^{j-j^\prime}-1}  &= \omega^{(i^\prime-i)p^{k-j^\prime}}\left(1+a^{-1}ub\right). 
		\end{align*}
		The above equation has a solution if and only if  $\omega^{(i^\prime-i)p^{k-j^\prime}}\left(1+a^{-1}ub\right)\in \langle \omega^{p^{j-j^\prime}-1}\rangle$. 
		Consequently, for any $u\in \mathbb{F}_q$, a solution in $a\in \mathbb{F}_q$ exists. In other words, if \eqref{eq:matrix} has a solution in $b\in \mathbb{F}_q$, then $a\in \mathbb{F}_q$ always exists.
		
		Therefore, for every $v=\omega^t$ such that $t\equiv (i^\prime-i)p^{k-j^\prime} \pmod{p^{j^\prime-j}-1}$ and $t\not\equiv 0 \pmod{p^{j^\prime- j}-1}$, the element $\omega^{(i^\prime-i)p^{k-j}}
		\varphi^{j^\prime-j}
		\begin{bmatrix}
			1 & u\\
			0 & v
		\end{bmatrix}$ is not a derangement, for any $u\in \mathbb{F}_q$. 
		In particular, there are exactly 
		\begin{align*}
			\left|\langle \omega^{p^{j^\prime -j}-1}\rangle\right| = \frac{q-1}{\gcd\left(q-1,p^{j^\prime - j}-1\right)}=\frac{q-1}{p^{\gcd\left(k,{j^\prime - j}\right)}-1}
		\end{align*}
		choices for $0\leq t\leq q-2$, or equivalently such $v\in \mathbb{F}_q^*$. Hence, there are $q^2-q - \frac{q-1}{p^{\gcd\left(k,{j^\prime - j}\right)}-1}q$ choices of $u\in \mathbb{F}_q$ and $v\in \mathbb{F}_q^*$ such that no point is fixed. 
		We conclude in the latter case that a vertex in $V_{i,j}$ is adjacent to $q^2-q - \frac{q-1}{p^{\gcd\left(k,{j^\prime - j}\right)}-1}q$ vertices in $V_{i^\prime,j^\prime}$.
	\end{proof}
	\begin{rmk}
		We note that Lemma~\ref{lem:across} implicitly implies that the characteristic $p$ is odd. Indeed, if $p=2$, then for any $j^\prime$ and $j$ such that $j^\prime-j$ and $k$ are coprime, we have
		\begin{align*}
			q^2-q-\left(\frac{q-1}{p^{\gcd\left(k,{j^\prime - j}\right)}-1}\right)q = 0.
		\end{align*}
		Hence in characteristic $2$, Lemma~\ref{lem:across} is false since we would still get a coclique in this case, even when $i\not\equiv i^\prime  \pmod{p^{j^\prime - j}}$. In the next lemma, we state the analogue of Lemma~\ref{lem:across} in characteristic $2$.
	\end{rmk}
	\begin{lem}
		Assume that $p=2$. For any $0\leq i,i^\prime\leq q-2$ and for any distinct $0\leq j<j^\prime \leq k-1$, the subgraph of $\Gamma_q$ induced by $V_{i,j} \cup V_{i^\prime,j^\prime}$ is a coclique if and only if $i^{\prime}\equiv i \pmod {p^{j^\prime-j}-1}$ or $\gcd(j^\prime -j,k) = 1 $. If $V_{i,j} \cup V_{i^\prime,j^\prime}$ is not a coclique of $\Gamma_q$, then a vertex of $V_{i,j}$ is adjacent to exactly 
		\begin{align*}
			q^2-q-\left(\frac{q-1}{p^{\gcd(k,j^\prime - j)}-1}\right)q
		\end{align*}
		vertices in $V_{i^\prime,j^\prime}$.\label{lem:across-even}
	\end{lem}
	
	Recall that for any $0\leq t\leq q-2$ and $0\leq j\leq k-1$, 
	\begin{align*}
		H_{t,j} = 
		\left\{
		\omega^i\varphi^j
		\begin{bmatrix}
			1& z\\
			0& \omega^{t+ip^{k-j}}
		\end{bmatrix}
		:
		0\leq j\leq k-1, z\in \mathbb{F}_q
		\right\}
		=\left\{
		\varphi^j
		\begin{bmatrix}
			x& z\\
			0& \omega^{t}
		\end{bmatrix}
		:
		x\in \mathbb{F}_q^*, z \in \mathbb{F}_q
		\right\}.
	\end{align*}
	
	\begin{lem}
		Assume that $p$ is an odd prime. For any $0\leq t,t^\prime\leq q-2$ and for any distinct $0\leq j,j^\prime \leq k-1$.
		\begin{enumerate}[(i)]
			\item The subgraph of $\Gamma_q$ induced by $H_{t,j} \cup H_{t^\prime,j^\prime}$ is a coclique if and only if $t^{\prime}\equiv t \pmod {p^{j^\prime-j}-1}$.\label{first}
			 \item If $H_{t,j} \cup H_{t^\prime,j^\prime}$ is not a coclique then a vertex of $H_{i,j}$ is adjacent to exactly
			 \begin{align*}
			 	q^2-q-\left(\frac{q-1}{p^{\gcd(j-j^\prime,k)}-1}\right)q
			 \end{align*}
			 vertices of $H_{i^\prime,j^\prime}$.\label{second}
			\end{enumerate}
			\label{lem:across-horizontal}
	\end{lem}
	\begin{proof}
		\eqref{first} Assume that $H_{t,j} \cup H_{t^\prime,j^\prime}$ is a coclique. Consider the elements $\omega^i\varphi^j \begin{bmatrix}
			1 & z\\
			0 & \omega^{t-ip^{k-j}}
		\end{bmatrix} \in H_{t,j}$ and $\omega^{i^\prime} \varphi^{j^\prime}\begin{bmatrix}
		1 & z^\prime\\
		0 & \omega^{t^\prime - i^\prime p^{k-j^\prime}}
		\end{bmatrix} \in H_{t^\prime,j^\prime}$, for some $0\leq i,i^\prime\leq q-2$. We have
		\begin{align*}
			\left(\omega^i\varphi^j \begin{bmatrix}
				1 & z\\
				0 & \omega^{t-ip^{k-j}}
			\end{bmatrix}\right)^{-1}\left(\omega^{i^\prime} \varphi^{j^\prime}\begin{bmatrix}
			1 & z^\prime\\
			0 & \omega^{t^\prime - i^\prime p^{k-j^\prime}}
			\end{bmatrix}\right)
			&=
			\begin{bmatrix}
				1 & z\\
				0 & \omega^{t-ip^{k-j}}
			\end{bmatrix}^{-1}
			\varphi^{-j}\omega^{i^\prime-i}\varphi^{j^\prime}\begin{bmatrix}
				1 & z^\prime\\
				0 & \omega^{t^\prime - i^\prime p^{k-j^\prime}}
			\end{bmatrix}\\
			&=
			\omega^{\left(i^\prime-i\right)p^{k-j}}
			\varphi^{j^\prime-j}
			\begin{bmatrix}
				1 & u\\
				0 & \omega^{  t^\prime-tp^{j-j^\prime}+ip^{k-j^\prime} - i^\prime p^{k-j^\prime}}
			\end{bmatrix}
		\end{align*}
		for some $u\in \mathbb{F}_q$. If this element fixes $[a\ b]^t$, then we have
		\begin{align}
			\begin{cases}
				a^{p^{j-j^\prime}} &= \omega^{(i^\prime - i)p^{k-j^\prime}}a +ub\\
				b^{p^{j-j^\prime}} &= b\omega^{t^\prime - tp^{j-j^\prime}}.
			\end{cases}\label{eq:horizontal}
		\end{align}
		We may distinguish whether $b = 0$ or $b \neq 0$. If $b\neq 0$, then $t^\prime \equiv tp^{{j-j^\prime}} \equiv t \pmod{p^{j-j^\prime}-1}$. If $b = 0$, then we must have $i^\prime - i \equiv 0 \pmod {p^{j-j^\prime}-1}$. By choosing $i$ and $i^\prime$ such that $i\not \equiv i^\prime \pmod{p^{j-j^\prime}-1}$ (these always exist as long as the characteristic is odd), edges occur in  $H_{t,j} \cup H_{t^\prime,j^\prime}$, which is impossible. Thus, we must have $t^\prime \equiv t \pmod{p^{j-j^\prime}-1}$.
		
		Conversely, assume that $t^\prime \equiv t \pmod{p^{j-j^\prime}-1}$, and suppose that there is an edge between $H_{t,j}$ and $H_{t^\prime,j^\prime}$. So, there exists elements $\omega^i\varphi^j \begin{bmatrix}
			1 & z\\
			0 & \omega^{t-ip^{k-j}}
		\end{bmatrix} \in H_{t,j}$ and $\omega^{i^\prime} \varphi^{j^\prime}\begin{bmatrix}
			1 & z^\prime\\
			0 & \omega^{t^\prime - i^\prime p^{k-j^\prime}}
		\end{bmatrix} \in H_{t^\prime,j^\prime}$, for some $0\leq i,i^\prime\leq q-2$, such that
		\begin{align*}
			\left(\omega^i\varphi^j \begin{bmatrix}
				1 & z\\
				0 & \omega^{t-ip^{k-j}}
			\end{bmatrix}\right)^{-1}\left(\omega^{i^\prime} \varphi^{j^\prime}\begin{bmatrix}
				1 & z^\prime\\
				0 & \omega^{t^\prime - i^\prime p^{k-j^\prime}}
			\end{bmatrix}\right)
			&=
			\omega^{\left(i^\prime-i\right)p^{k-j}}
			\varphi^{j^\prime-j}
			\begin{bmatrix}
				1 & u\\
				0 & \omega^{  t^\prime-tp^{j-j^\prime}+ip^{k-j^\prime} - i^\prime p^{k-j^\prime}}
			\end{bmatrix}
		\end{align*}
		for some $u\in \mathbb{F}_q$, has no fixed point. One can immediately see that $i^\prime \not \equiv i \pmod{p^{j-j^\prime}-1}$ otherwise a solution exists by taking $b = 0$. Similarly, $t^\prime \not \equiv t \pmod{p^{j-j^\prime}-1}$, otherwise a solution in $b\in \mathbb{F}_q$ in \eqref{eq:horizontal} exists, which implies the existence of $a\in \mathbb{F}_q$ as well. Hence, we get a contradiction, and $H_{t,j}\cup H_{t^\prime,j^\prime}$ must be a coclique in $\Gamma_q$.
		
		\eqref{second} Assume that $H_{t,j} \cup H_{t^\prime,j^\prime}$ is not a coclique, that is, $t\not\equiv t^\prime \pmod{p^{j-j^\prime}-1}$. Again consider the elements $g = \omega^i\varphi^j \begin{bmatrix}
			1 & z\\
			0 & \omega^{t-ip^{k-j}}
		\end{bmatrix} \in H_{t,j}$ and $g^\prime = \omega^{i^\prime} \varphi^{j^\prime}\begin{bmatrix}
			1 & z^\prime\\
			0 & \omega^{t^\prime - i^\prime p^{k-j^\prime}}
		\end{bmatrix} \in H_{t^\prime,j^\prime}$, for some $0\leq i,i^\prime\leq q-2$. Assume that $[a \ b]^t$ if fixed by $g^{-1}g^\prime$. Using \eqref{eq:horizontal}, we can see that if $b\neq 0$ then we must have $t^\prime \equiv t\pmod {p^{j-j^\prime}-1}$, which is impossible. We deduce that, $b= 0$ and $i^\prime \equiv i \pmod{p^{j-j^\prime}-1}$. For any fixed $i$, there are exactly $\frac{q-1}{p^{\gcd(j-j^\prime, k)}-1}$ elements $i^\prime$ that satisfy the previous modular equation. For any such $i^\prime$, we may freely choose $z^\prime \in \mathbb{F}_q$ to obtain non-adjacent vertices. We conclude that a vertex of $H_{t,j}$ is adjacent
		\begin{align*}
			q^2-q-\left(\frac{q-1}{p^{\gcd(j-j^\prime,k)}-1}\right)q
		\end{align*}
		vertices of $H_{t^\prime,j^\prime}$.
	\end{proof}
	
	As we can see in the previous lemma, the statement does not hold in even characteristic. We give the characteristic $2$ analogue of the previous lemma.
	\begin{lem}
		Assume that $p=2$. For any $0\leq i,i^\prime\leq q-2$ and for any $0\leq j<j^\prime \leq k-1$, the subgraph of $\Gamma_q$ induced by $H_{t,j} \cup H_{t^\prime,j^\prime}$ is a coclique if and only if $t^{\prime}\equiv t \pmod {p^{j^\prime-j}-1}$ or $\gcd(j^\prime -j,k) = 1 $. If $H_{i,j} \cup H_{i^\prime,j^\prime}$ is not a coclique of $\Gamma_q$, then a vertex of $H_{i,j}$ is adjacent to exactly 
		\begin{align*}
			q^2-q-\left(\frac{q-1}{p^{\gcd(k,j^\prime - j)}-1}\right)q
		\end{align*}
		vertices in $H_{i^\prime,j^\prime}$.\label{lem:across-even-horizontal}
	\end{lem}
	
	Now, we are ready to give the proof of Theorem~\ref{thm1}.
	\section{The maximum cocliques of $\gammal{2}{q}$ in odd characteristic}\label{sect:cocliques-odd}
	Assume that $p$ is an odd prime and $q = p^k$. From Lemma~\ref{lem:EKR}, we know that $\gammal{2}{q}$ has the EKR property. As we have also seen in \eqref{eq:H}, a maximum intersecting set of $\gammal{2}{q}$ may be assumed to be contained in $\mathcal{H}_q$. Recall also that $(V_{i,j})$ is a partition of $\mathcal{H}_q$.
	
	Let $\mathcal{F}\subset \gammal{2}{q}$ be an intersecting set of maximum size contained in $\mathcal{H}_q$. We define $\mathcal{F}_{i,j} := \mathcal{F} \cap V_{i,j}$ for any $0\leq i\leq q-2$ and $0\leq j\leq k-1$. Moreover, for any $0\leq j\leq q-2$, we define the \itbf{$j$-th layer} of $\mathcal{F}$ to be $\mathcal{F}^{(j)}:=\mathcal{F}_{0,j} \cup \mathcal{F}_{1,j} \cup \ldots \cup \mathcal{F}_{q-2,j}$. Since the identity belongs to $\mathcal{F}$, note that $\mathcal{F}_{0,0} \neq \varnothing$.
		
		\vspace*{0.5cm}
		\noindent{\bf Claim~1.} {\it For any $0\leq j\leq k-1$, we have $|\mathcal{F}^{(j)}|=q(q-1)$.}
		\begin{proof}[Proof of Claim~1]
			Let $0\leq j\leq k-1.$  By Corollary~\ref{cor:one-layer}, the size of a maximum coclique in the subgraph of $\Gamma_q$ induced by $\mathcal{F}^{(j)} = V_{0,j}\cup V_{1,j}\cup \ldots\cup V_{q-2,j}$ is $q(q-1)$. Hence, $|\mathcal{F}^{(j)}|\leq q(q-1)$. Since $\mathcal{F}$ is a maximum intersecting set, we must have $|\mathcal{F}^{(j)}|= q(q-1)$ otherwise there would be $0\leq j^\prime \leq k-1$ such that $|\mathcal{F}^{(j^\prime)}|=|\mathcal{F}_{0,j^\prime} \cup \mathcal{F}_{1,j^\prime} \cup \ldots \cup \mathcal{F}_{q-2,j^\prime}|> q(q-1)$.
		\end{proof}
		The above claim tells us that such cocliques are balanced in each layer of $\mathcal{F}$ (corresponding to $0\leq j\leq q-2$) due to the fact that they are maximum. From Lemma~\ref{lem:same-layer}, we know that a coclique of size $q(q-1)$ is one of two types. For any $0\leq j\leq k-1$, since $|\mathcal{F}^{(j)}| = |\mathcal{F}_{0,j} \cup \mathcal{F}_{1,j} \cup \ldots \cup \mathcal{F}_{q-2,j}|=q(q-1)$, either there exists a unique $0\leq i \leq q-2$ such that $\mathcal{F}^{(j)} = \mathcal{F}_{i,j} = V_{i,j}$ or for every $0\leq i\leq q-2$, $|\mathcal{F}_{i,j}| = q$. If $\mathcal{F}_{i,j} = V_{i,j}$, then we say that $\mathcal{F}^{(j)}$ is a $j$-\itbf{vertical coclique}. If $|\mathcal{F}_{i,j}| = q$, for every $0\leq i\leq q-2$, then we say that $\mathcal{F}^{(j)}$ is a $j$-\itbf{horizontal coclique}.
		
		\vspace*{0.5cm}
		\noindent {\bf Claim~2.} {\it $\mathcal{F}$ cannot simultaneously contain a $j$-vertical coclique and a $j^\prime$-horizontal coclique, for $0\leq j,j^\prime\leq k-1$.}
		
		\begin{proof}[Proof of Claim~2]
			Let us prove this by contradiction. Assume that $\mathcal{F}^{(j)}$ is a vertical coclique and $\mathcal{F}^{(j^\prime)}$ is a horizontal coclique. Therefore, there exists $0\leq i\leq q-2$ such that $\mathcal{F}^{(j)} = \mathcal{F}_{i,j} = V_{i,j}$. Since $\mathcal{F}^{(j^\prime)}$ is a horizontal coclique, $|\mathcal{F}_{i^\prime,j^\prime}| = q$ for any $0\leq i^\prime\leq q-2$. Fix $0\leq i^\prime\leq q-2$ such that $i^\prime - i \not \equiv 0 \pmod {p^{j-j^\prime}-1}$. By Lemma~\ref{lem:across}, we know that $V_{i,j}\cup V_{i^\prime,j^\prime} = \mathcal{F}_{i,j}\cup V_{i^\prime,j^\prime}$ is not a coclique. Precisely, any vertex of $V_{i^\prime,j^\prime}$, and in particular any vertex of $\mathcal{F}_{i^\prime,j^\prime}$, is adjacent to $q^2-q - \frac{q-1}{p^{\gcd(j-j^\prime,k)-1}}q$ of $V_{i,j} = \mathcal{F}_{i,j}$. Therefore, we obtain a contradiction.
		\end{proof}
		
		\vspace*{0.5cm}
		\noindent {\bf Claim~3.} {\it  If $\mathcal{F}^{(0)}$ is a horizontal coclique, then 
			$
				\mathcal{F} = H_{0,0} \cup H_{t_1(p-1),1} \cup \ldots \cup H_{t_j(p^j-1),j} \cup \ldots \cup H_{t_{k-1}(p^{k-1}-1),k-1},
			$
			such that 
			\begin{align*}
				\begin{cases}
					t_0 = 0&\\
					0\leq t_j\leq \frac{q-1}{p^{\gcd(j,k)}-1}-1 &\mbox{ for }\ 0\leq j\leq k-1\\
					t_j(p^j-1) \equiv t_{j^\prime}(p^{j^\prime}-1) \pmod{p^{j^\prime-j}-1}&\mbox{ for any $0\leq j<j^\prime \leq k-1$.}
				\end{cases}
		\end{align*}}
		\begin{proof}[Proof of Claim~3]
			Since $\mathcal{F}^{(0)}$ is a horizontal coclique and $\mathcal{F}_{0,0} \neq \varnothing$, by Lemma~\ref{lem:same-layer}, we deduce that $|\mathcal{F}_{0,0}| = q$. Moreover, we must have $\mathcal{F}^{(0)} = H_{0,0}$ and $\mathcal{F}_{i,0} = \omega^i\mathcal{E}(\omega^{-i})$ for $0\leq i\leq q-2$ (see \eqref{eq:epsilon} for the definition of $\varepsilon(\omega^{-i})$). By Claim~2, we know that $\mathcal{F}^{(j)}$ is a horizontal coclique for all $0\leq j\leq k-1$. Hence, we must have 
			\begin{align*}
				\mathcal{F} = H_{0,0} \cup H_{i_1,1} \cup \ldots \cup H_{i_{k-1},k-1}
			\end{align*}
			for some $0\leq i_1,i_2,\ldots,i_{k-1} \leq q-2$. By Lemma~\ref{lem:across-horizontal}, we know that $i_j\equiv 0 \pmod {p^{j}-1}$ for any $0\leq j\leq k-1$. Hence, for any $0\leq j\leq k-1$ there exists $0\leq t_j \leq \frac{q-1}{p^{\gcd(j,k)}-1}-1$ such that
			$i_j = t_j(p^j-1)$. Further, by Lemma~\ref{lem:across-horizontal}, we know that $t_j(p^j-1) \equiv t_{j^\prime} (p^{j^\prime}-1) \pmod {p^{j^\prime - j}-1}$ for any $0\leq j,j^\prime \leq k-1$. This completes the proof.
		\end{proof}
		
		\vspace*{0.5cm}
		\noindent {\bf Claim~4.} {\it  If $\mathcal{F}^{(0)}$ is a vertical coclique, then 
		$
			\mathcal{F} = V_{0,0} \cup V_{r_1(p-1),1} \cup \ldots \cup V_{r_j(p^j-1),i} \cup \ldots \cup V_{r_{k-1}(p^{k-1}-1),k-1},
		$
		such that 
		\begin{align*}
			\begin{cases}
				r_0 = 0&\\
				0\leq r_j\leq \frac{q-1}{p^{\gcd(j,k)}-1}-1 &\mbox{ for }\ 0\leq j\leq k-1\\
				r_j(p^j-1) \equiv r_{j^\prime}(p^{j^\prime}-1) \pmod{p^{j^\prime-j}-1}&\mbox{ for any $0\leq j<j^\prime \leq k-1$.}
			\end{cases}
			\end{align*}}
		\begin{proof}[Proof of Claim~4]
			Since $\mathcal{F}^{(0)}$ is a vertical coclique, we know from Claim~2 that $\mathcal{F}^{(j)}$ are vertical cocliques for all $0\leq j\leq k-1$. Moreover, there exists $0\leq i_1,i_2,\ldots,i_{k-1}\leq k-1$ such that $\mathcal{F}_{i,j} = V_{i_j,j}$ and thus
			\begin{align*}
				\mathcal{F} = V_{0,0} \cup V_{i_1,1}\cup \ldots V_{i_{k-1},k-1}. 
			\end{align*}
			By Lemma~\ref{lem:across}, we know that $i_j\equiv 0 \pmod {p^j-1}$, for any $0\leq j\leq k-1$. For any $0\leq j\leq k-1$, there exists $0\leq r_j \leq \frac{q-1}{p^{\gcd(j,k)}-1}-1$ such that $i_j = r_j(p^j-1)$. By Lemma~\ref{lem:across} again, we conclude that for any $0\leq j,j^\prime \leq k-1$ we also have $r_j(p^j-1) \equiv r_{j^\prime} (p^{j^\prime}-1) \pmod{p^{j^\prime -p}-1}$.
		\end{proof}
		
		This completes the proof of Theorem~\ref{thm1}. The proof of Corollary~\ref{cor:two-layers} also follows easily from this. If $\mathcal{F}$ is a maximum coclique containing the identity, then by Claim~3 and Claim~4, either $\mathcal{F} = H_{0,0} \cup H_{t_1(p-1),1}$ or $\mathcal{F} = V_{0,0} \cup V_{r_1(p-1),1}$. It is not hard to see that $ V_{0,0} \cup V_{r_1(p-1),1} = \mathcal{N}_q \rtimes \langle\omega^{r_1(p-1)}\varphi \rangle$. Similarly, $H_{0,0} \cup H_{t_1(p-1),1} = \left\{ \omega^i\mathcal{E}(\omega^{-i}) :\ 0\leq i\leq q-2\right\} \rtimes \langle \omega^{t_1(p-1)}\varphi \rangle$. These are conjugate to the stabilizer of $e_1$ and $e_2+\langle e_1\rangle$ respectively.
	
	\section{The maximum cocliques of $\gammal{2}{q}$ In even characteristic}\label{sect:cocliques-even}
	
	Throughout this section, we assume that $q = 2^k$ for some integer $k\geq 1$. Let $\mathcal{F}\subset \gammal{2}{q}$ be an intersecting set of maximum size contained in $\mathcal{H}_q$. We define $\mathcal{F}_{i,j} := \mathcal{F} \cap V_{i,j}$ for any $0\leq i\leq q-2$ and $0\leq j\leq k-1$. Moreover, for any $0\leq j\leq q-2$, we define the \itbf{$j$-th layer} of $\mathcal{F}$ to be $\mathcal{F}^{(j)}:=\mathcal{F}_{0,j} \cup \mathcal{F}_{1,j} \cup \ldots \cup \mathcal{F}_{q-2,j}$. Since the identity belongs to $\mathcal{F}$, note that $\mathcal{F}_{0,0} \neq \varnothing$.
	
	\begin{thm}
		Assume that $q = 2^k$ such that $k$ is not a prime power. The group $\gammal{2}{q}$ acting on non-zero vectors of $\mathbb{F}_q^2$ has the EKR property, and in particular any intersecting set has size at most $kq(q-1)$.  In addition, an intersecting set of maximum size in $\gammal{2}{q}$ is one of the following two types.
		\begin{enumerate}[1)]
			\item Vertical types: 
			\begin{align*}
				V_{0,0} \cup V_{r_1,1} \cup \ldots \cup V_{r_j,j} \cup \ldots \cup V_{r_{k-1},k-1},
			\end{align*}
			such that 
			\begin{align*}
				\begin{cases}
					r_0 = 0 &\\
					0\leq  r_j\leq q-2& \mbox{ for any }0\leq j\leq k-1,\\
					\gcd(j^\prime-j,k) = 1 \mbox{ or }
					r_j \equiv r_{j^\prime} \pmod{p^{j^\prime-j}-1} & \mbox{ for any $0\leq j<j^\prime \leq k-1$.}
				\end{cases}
			\end{align*}
			
			\item Horizontal types:
			\begin{align*}
				H_{0,0} \cup H_{t_1,1} \cup \ldots \cup H_{t_i,i} \cup \ldots \cup H_{t_{k-1},k-1},
			\end{align*}
			such that 
			\begin{align*}
				\begin{cases}
					t_0 = 0 &\\
					0\leq t_j\leq q-2 &\mbox{ for }\ 0\leq j\leq k-1\\
					\gcd(j^\prime-j,k) = 1 \mbox{ or }t_j \equiv t_{j^\prime} \pmod{p^{j^\prime-j}-1}&\mbox{ for any $0\leq j<j^\prime \leq k-1$.}
				\end{cases}
			\end{align*}
		\end{enumerate}
		\label{thm3}
	\end{thm}

	\begin{thm}
		Assume that $q = 2^k$ such that $k = s^\ell$ is a prime power. The group $\gammal{2}{q}$ acting on non-zero vectors of $\mathbb{F}_q^2$ has the EKR property, and in particular any intersecting set has size at most $kq(q-1)$.  In addition, if an intersecting set of maximum size in $\gammal{2}{q}$, then 
		\begin{align*}
			\mathcal{F} = F_0\cup \omega^{i_1}\varphi^s F_1 \cup \ldots \cup \omega^{i_{s-1}}\varphi^{s-1}F_{\ell-1}
		\end{align*}
		where $F_j$, for every $0\leq j\leq \ell - 1$, is of the following types.
		\begin{enumerate}[1)]
			\item Vertical types. 
			\begin{align*}
				V_{0,0} \cup V_{r_1,s} \cup \ldots \cup V_{r_j,sj} \cup \ldots \cup V_{r_{(s^{\ell-1}-1)},s\left(s^{\ell-1}-1\right)},
			\end{align*}
			such that 
			\begin{align*}
				\begin{cases}
					r_0 = 0 & \\
					0\leq  r_j\leq q-2& \mbox{ for any }0\leq j\leq s^{\ell-1}-1,\\
					r_j \equiv r_{j^\prime} \pmod{p^{s(j^\prime-j)}-1} & \mbox{ for any $0\leq j<j^\prime \leq s^{\ell-1}-1$.}
				\end{cases}
			\end{align*}
			
			\item Horizontal types.
			\begin{align*}
				 H_{0,0} \cup H_{t_1,s} \cup \ldots \cup H_{t_j,sj} \cup \ldots \cup H_{t_{\ell-1},s\left(s^{\ell-1}-1\right)},
			\end{align*}
			such that 
			\begin{align*}
				\begin{cases}
					t_0 =0 & \\
					0\leq t_j\leq q-2 &\mbox{ for }\ 0\leq j\leq s^{\ell-1}-1\\
					t_j \equiv t_{j^\prime} \pmod{p^{s(j^\prime-j)}-1}&\mbox{ for any $0\leq j<j^\prime \leq s^{\ell-1}-1$.}
				\end{cases}
			\end{align*}
		\end{enumerate}
		\label{thm4}
	\end{thm}
	
	Assume that $q = 2^k$. Since a Singer subgroup of $\gl{2}{q}$ is a regular subgroup of $\gammal{2}{q}$, it is clear the latter has the EKR property, and the largest intersecting sets have size $kq(q-1)$. Let $\Gamma_q$ be the subgraph of the derangement graph of $\Gamma_q$ induced by $\mathcal{H}_q$. 

	If $k$ is not a prime power, then there are two distinct prime numbers $j$ and $j^\prime$ that are not coprime with $k$ and $\langle \varphi^j,\varphi^{j^\prime} \rangle = \langle\varphi \rangle$. The latter implies that $\Gamma_q$ is connected due to Lemma~\ref{lem:across-even}. If $k$ is a prime power, then $\Gamma_q$ is disconnected since there is a unique maximal subgroup in $\langle \varphi\rangle$ which contains all the non-generators of $\langle \varphi\rangle$. Therefore, we will distinguish these two cases. If $k$ is not a prime power, then the proof of Theorem~\ref{thm3} is given in Section~\ref{subsect1}. In the case where $k$ is a prime power, the proof of Theorem~\ref{thm4} is given in Section~\ref{subsect2}. 
	
	Before proceeding with these cases, we note that the following claims also hold in even characteristic (the proofs are identical to the ones in odd characteristic so we omit them).
	
	\vspace*{0.5cm}
	\noindent{\bf Claim~5.} {\it For any $0\leq j\leq k-1$, we have $|\mathcal{F}^{(j)}|=q(q-1)$.}
	
	\vspace*{0.5cm}
	\noindent {\bf Claim~6.} {\it If $\Gamma$ is a component of $\Gamma_q$, then $\mathcal{F}\cap V(\Gamma)$ cannot contain a $j$-vertical coclique and a $j^\prime$-horizontal coclique, for $0\leq j,j^\prime\leq k-1$.}

	\subsection{$k$ is not a prime power}\label{subsect1}
	As we saw before, in this case $\Gamma_q$ is connected. The proof of Theorem~\ref{thm3} follows from the next claims.
	
	\vspace*{0.5cm}
	\noindent {\bf Claim~7.} {\it  If $\mathcal{F}^{(0)}$ is a horizontal coclique, then 
		$
		\mathcal{F} = H_{0,0} \cup H_{t_1,1} \cup \ldots \cup H_{t_j,j} \cup \ldots \cup H_{t_{k-1},k-1},
		$
		such that 
		\begin{align*}
			\begin{cases}
				t_0 = 0 &\\
				0\leq t_j\leq q-2 &\mbox{ for }\ 0\leq j\leq k-1\\
				\gcd(j^\prime-j,k) = 1 \mbox{ or } t_j \equiv t_{j^\prime} \pmod{p^{j^\prime-j}-1}&\mbox{ for any $0\leq j<j^\prime \leq k-1$.}
			\end{cases}
	\end{align*}}
	\begin{proof}[Proof of Claim~7]
		Since $\mathcal{F}^{(0)}$ is a horizontal coclique and $\mathcal{F}_{0,0} \neq \varnothing$, by Lemma~\ref{lem:same-layer}, we deduce that$|\mathcal{F}_{0,0}| = q$. Moreover, we must have $\mathcal{F}^{(0)} = H_{0,0}$ and so $\mathcal{F}_{i,0} = \omega^{i}\mathcal{E}(\omega^{-i})$ for $0\leq i\leq k-1$. By Claim~6, we know that $\mathcal{F}^{(j)}$ is a horizontal coclique for all $0\leq j\leq k-1$. Hence, we must have 
		\begin{align*}
			\mathcal{F} = H_{0,0} \cup H_{t_1,1} \cup \ldots \cup H_{t_{k-1},k-1}
		\end{align*}
		for some $0\leq t_1,t_2,\ldots,t_{k-1} \leq q-2$. The result follows immediately from Lemma~\ref{lem:across-even-horizontal}.
	\end{proof}
	\vspace*{0.5cm}
	\noindent {\bf Claim~8.} {\it  If $\mathcal{F}^{(0)}$ is a vertical coclique, then 
		$
			\mathcal{F} = V_{0,0} \cup V_{r_1,1} \cup \ldots \cup V_{r_j,j} \cup \ldots \cup V_{r_{k-1},k-1},
		$
		such that 
		\begin{align*}
			\begin{cases}
				r_0 = 0& \\
				0\leq  r_j\leq q-2& \mbox{ for any }1\leq j\leq k-1,\\
				\gcd(j^\prime-j,k) = 1 \mbox{ or }
				r_j \equiv r_{j^\prime} \pmod{p^{j^\prime-j}-1} & \mbox{ for any $0\leq j<j^\prime \leq k-1$.}
			\end{cases}
		\end{align*}}
	\begin{proof}[Proof of Claim~8]
		Since $\mathcal{F}^{(0)}$ is a vertical coclique, we know from Claim~2 that $\mathcal{F}^{(j)}$ are vertical cocliques for all $0\leq j\leq k-1$. Moreover, there exists $0\leq r_1,r_2,\ldots,r_{k-1}\leq k-1$ such that $\mathcal{F}_{i,j} = V_{r_j,j}$ and thus
		\begin{align*}
			\mathcal{F} = V_{0,0} \cup V_{i_1,1}\cup \ldots V_{i_{k-1},k-1}. 
		\end{align*}
		By Lemma~\ref{lem:across-even}, the result also follows immediately.
	\end{proof}
	
	 This completes the proof of Theorem~\ref{thm3}.
	\subsection{$k$ is a prime power}\label{subsect2}
	Assume that $k = s^\ell$, where $s$ is a prime and $\ell \geq 2$ is an integer. In this case, we know that $\Gamma_q$ is disconnected. Any non-generator of $\langle \varphi \rangle$ is contained in the unique maximal subgroup $\langle \varphi^{s}\rangle$. From this, we deduce that there are exactly $s$ components in $\Gamma_q$. As the set $\{id,\varphi,\varphi^2,\ldots,\varphi^{{s-1}}\}$ is a left transversal\footnote{that is, a system of distinct representatives of the cosets of $\langle \varphi^s\rangle$ in $\langle \varphi\rangle$} of $\langle \varphi^s\rangle$ in $\langle \varphi\rangle$, we may assume that $\Gamma_q^{(0)},\Gamma_q^{(1)},\Gamma_q^{(2)},\ldots,\Gamma_q^{(s-1)}$ are the components of $\Gamma_q$ which respectively correspond to the elements $id,\varphi,\varphi^2,\ldots,\varphi^{(s-1)}$. In fact, $\Gamma_q^{(t)}$, for any $1\leq t\leq s-1$, is the subgraph induced by all $V_{i,j}$ such that $\gcd(j-t,k) \neq 1$. From this, we know that $\Gamma_q^{(0)}$ is the subgraph that contains the identity element of the group. If $V_{i,j}$ and $V_{i^\prime,j^\prime}$ are contained in $\Gamma_q^{(t)}$, then $s\mid \gcd(j-t,k)$ and $s\mid \gcd(j^\prime-t,k)$, and thus we must have $\gcd(j-j^\prime,k)\neq 1$.
	
	\vspace*{0.5cm}
	\noindent {\bf Claim~9.} {\it  If $\mathcal{F}^{(0)} \subset V\left(\Gamma_q^{(0)}\right)$ is a horizontal coclique, then 
		\begin{align*}
			\mathcal{F}\cap V\left(\Gamma_q^{(0)}\right) = H_{0,0} \cup H_{t_1,s} \cup \ldots \cup H_{t_j,js} \cup \ldots \cup H_{t_{s(s^{\ell-1}-1)},s\left(s^{\ell-1}-1\right)},
		\end{align*}
		such that 
		\begin{align*}
			\begin{cases}
				t_0 = 0 &\\
				0\leq t_j\leq q-2 &\mbox{ for }\ 0\leq j\leq \left(s^{\ell-1}-1\right)\\
				t_j \equiv t_{j^\prime} \pmod{p^{s(j^\prime-j)}-1}&\mbox{ for any $0\leq j<j^\prime \leq \left(s^{\ell-1}-1\right)$.}
			\end{cases}
	\end{align*}}
	\begin{proof}[Proof of Claim~9]
		Since $\mathcal{F}^{(0)}$ is a horizontal coclique and $\mathcal{F}_{0,0} \neq \varnothing$, by Lemma~\ref{lem:same-layer}, we deduce that$|\mathcal{F}_{0,0}| = q$. Moreover, we must have $\mathcal{F}^{(0)} = H_{0,0}$ and $\mathcal{F}_{i,0} = \omega^{i}\mathcal{E}(\omega^{-i})$ for $0\leq i\leq k-1$. By Claim~6, we know that $\mathcal{F}^{(j)}$ is a horizontal coclique for all $0\leq j\leq (s^{\ell-1}-1)$. As the number of layers of $\Gamma_q^{(0)}$ is $\frac{s^\ell}{s} = s^{\ell-1}$, we must have 
		\begin{align*}
			\mathcal{F} = H_{0,0} \cup H_{t_1,s}\cup H_{t_2,2s} \cup \ldots \cup H_{t_{s^{\ell-1}-1},s\left(s^{\ell-1}-1\right)}
		\end{align*}
		for some $0\leq t_1,t_2,\ldots,t_{s^{\ell-1}-1} \leq q-2$. The result follows immediately from Lemma~\ref{lem:across-even-horizontal} and the fact that $\gcd(sj^\prime-sj,k)\neq 1$ for any $0\leq j<j^\prime\leq s^{\ell -1}-1$.
	\end{proof}
	\vspace*{0.5cm}
	\noindent {\bf Claim~10.} {\it  If $\mathcal{F}^{(0)} \subset V\left(\Gamma_q^{(0)}\right)$ is a vertical coclique, then 
		\begin{align*}
			\mathcal{F}\cap V \left(\Gamma_q^{(0)}\right) = V_{0,0} \cup V_{r_1,s} \cup \ldots \cup V_{r_j,sj} \cup \ldots \cup V_{r_{s^{\ell-1}-1},s(s^{\ell-1}-1)},
		\end{align*}
		such that 
		\begin{align*}
			\begin{cases}
				r_0 = 0 &\\
				0\leq  r_j\leq q-2& \mbox{ for any }0\leq j\leq (s^{\ell-1}-1),\\
				r_j \equiv r_{j^\prime} \pmod{p^{s(j^\prime-j)}-1} & \mbox{ for any $0\leq j<j^\prime \leq (s^{\ell-1}-1)$.}
			\end{cases}
	\end{align*}}
	\begin{proof}[Proof of Claim~10]
		Since $\mathcal{F}^{(0)}$ is a vertical coclique, we know from Claim~6 that $\mathcal{F}^{(j)}$ are vertical cocliques for all $0\leq j\leq (s^{\ell-1}-1)$. Moreover, there exists $0\leq r_1,r_2,\ldots,r_{(s^{\ell-1}-1)}\leq q-2$ such that $\mathcal{F}_{i,j} = V_{r_j,sj}$ and thus
		\begin{align*}
			\mathcal{F}\cap V\left(\Gamma_q^{(0)}\right) = V_{0,0} \cup V_{r_1,s}\cup \ldots V_{r_{(s^{\ell-1}-1)},s\left(s^{\ell-1}-1\right)}. 
		\end{align*}
		By Lemma~\ref{lem:across-even}, these indices must satisfy
		\begin{align*}
				r_j \equiv r_{j^\prime} \pmod{p^{s(j^\prime-j)}-1} & \mbox{ for any $0\leq j,j^\prime \leq (s^{\ell-1}-1)$.}
		\end{align*}
		The result then follows immediately.
	\end{proof}
	\vspace*{0.5cm}
	\noindent {\bf Claim~11.} {\it  If $0\leq j\leq s-1$, then 
		$
			\mathcal{F}\cap V \left(\Gamma_q^{(j)}\right) = \omega^i\varphi^j\mathcal{F}^\prime,
		$
		for some coclique $\mathcal{F}^{\prime}$ of maximum size in $\Gamma_q^{(0)}$ and $0\leq i\leq q-2$.}
	\begin{proof}[Proof of Claim~11]
		Since $\Gamma_q^{(0)}$ is isomorphic to $\Gamma_q^{(j)}$, a maximum coclique of $\Gamma_q^{(j)}$ containing $\varphi^{j}$ can be obtained by multiplying a maximum coclique of $\Gamma_q^{(0)}$ by $\varphi^j$, since such cocliques of $\Gamma_q^{(0)}$ are assumed to contain the identity element. A maximum coclique in $\Gamma_q^{(j)}$ need not contain $\varphi^j$, however, it must contain an element of the form $\omega^i\varphi^j$ for some $0\leq i\leq q-2$. Therefore, there exists a maximum coclique $\mathcal{F}^\prime$ of $\Gamma_q^{(0)}$ such that $\mathcal{F} = \omega^{i}\varphi^j\mathcal{F}^\prime$.
	\end{proof}
	
	Combining these three claims, we obtain the proof of Theorem~\ref{thm3}.
	
	\section{No Hilton-Milner sets in $\gl{2}{q}$}\label{sect:HM}
	
	In this section, we revisit the fact that all maximal intersecting sets of $\gl{2}{q}$ acting on non-zero vectors of $\mathbb{F}_q^2$ are maximum. This result was proved in \cite{ahanjideh2022largest} and \cite{maleki2021no}. We believe that our proof is much simpler than the proofs in these papers, so we reprove this result here.
	
	\begin{thm}
		If $\mathcal{F} \subset \gl{2}{q}$ is a maximal intersecting set, then it is a maximum intersecting set. In particular, there is no Hilton-Milner sets in $\gl{2}{q}$.\label{thm:HM-GL}
	\end{thm}
	
	In order to prove this result, we need the following lemmas.
	\begin{lem}
		If $\mathcal{F} \subset \gl{2}{q}$ is intersecting, then any three distinct elements of $\mathcal{F}$ fix exactly one element of $\pg{1}{q}$.\label{lem:base}
	\end{lem}
	\begin{proof}
		Let $A,B,C \in \mathcal{F}$ be distinct. Note that by multiplying with the right element, we may assume that $C = I_2$, and therefore $A$ and $B$ must have a fixed point. Let $u$ and $v$ be the respective elements of $\mathbb{F}_q^2$ fixed by $A$ and $B$. As there is a unique matrix $R\in \gl{2}{q}$ such that $Ru = e_1$ and $Rv= e_2$, by conjugating with $R$, we may also assume that $A$ fixes $e_1$ and $B$ fixes $e_2$. In this case, we have
		\begin{align*}
			A = \begin{bmatrix}
				1 & x\\
				0 & y
			\end{bmatrix} \mbox{ and }
			B = \begin{bmatrix}
				a & 0\\
				b& 1
			\end{bmatrix},
		\end{align*}
		for some $x,b\in \mathbb{F}_q$ and $y,a\in \mathbb{F}_q^*$. Note that $A$ fixes the line $\langle e_1\rangle$, and additionally the line
		\begin{align*}
			 \left\langle \begin{bmatrix}
				\frac{x}{y-1}\\
				1
			\end{bmatrix} \right\rangle
		\end{align*}
		 if $y\neq 1$.
		Similarly, $B$ fixes the line $\langle e_2\rangle$, and additionally the line
		\begin{align*}
			 \left\langle \begin{bmatrix}
				1\\
				\frac{b}{a-1}
			\end{bmatrix} \right\rangle
		\end{align*}
		if $a\neq 1$.
		
		Since $A^{-1}B$ must also fix a point, the matrix
		\begin{align*}
			A^{-1}B - I_2 = \begin{bmatrix}
				bxy^{-1} & -xy^{-1}\\
				by^{-1} & y^{-1}-1
			\end{bmatrix}
		\end{align*}
		has rank $1$. Therefore,
		\begin{align*}
			\det(A^{-1}B-I_2) = -bxy^{-1} = 0.
		\end{align*}
		Since $y^{-1} \neq 0$, we conclude that $x=0$ or $b = 0$, and so $A$ fixes $\langle e_2\rangle$ or $B$ fixes $\langle e_1\rangle$. This takes care of the existence of such elements in $\pg{1}{q}$.
		
		If $A$ fixes $\langle e_2\rangle$ and $B$ fixes $\langle e_1\rangle$, then $b=0$ and $x = 0$. Therefore, $A$ and $B$ are diagonal matrices. Since they are also intersecting it is not hard to see that one of them must be the identity, which is impossible.
		
		Therefore, either $A$ fixes $\langle e_2\rangle$ or $B$ fixes $\langle e_1\rangle$.
	\end{proof}
	\begin{lem}
		If $\mathcal{F} \subset \gl{2}{q}$ is intersecting, then there exists an element of $\pg{1}{q}$ that is fixed by every element of $\mathcal{F}$.\label{lem:fix-line}
	\end{lem}
	\begin{proof}
		Let $\mathcal{F}\subset \gl{2}{q}$ be intersecting.
		
		If $|\mathcal{F}| = 2$, then the statement trivially holds. If $|\mathcal{F}| = 3$, then the statement holds by Lemma~\ref{lem:base}. Assume that $|\mathcal{F}|\geq 4$. Let $A,B \in \mathcal{F}$. By Lemma~\ref{lem:base} and its proof, suppose that $\langle e_1\rangle$ is fixed by $\{I_2,A,B\}$ and that $B$ also fixes $\langle e_2\rangle$. From the uniqueness part of Lemma~\ref{lem:base}, we know that $A$ does not fix $\langle e_2\rangle$. 
		
		Now let $C\in \mathcal{F} \setminus \{ I_2,A,B\}$. Applying Lemma~\ref{lem:base} on $\{I_2,B,C\}$, we know that $C$ fixes one of $\langle e_1\rangle$ or $\langle e_2\rangle$.  Note that $\{A,B,C\}$ fixes $\langle w\rangle \in \pg{1}{q}$ which must be an eigenspace of $A$, $B$ and $C$, so $\langle \omega\rangle \in \left\{ \langle e_1\rangle, \langle e_2\rangle \right\}$.
		The case $\langle w\rangle = \langle e_2\rangle$ is impossible since $\langle e_2 \rangle$ is not fixed by $A$.  If $\langle w\rangle = \langle e_1\rangle$, then $C$ fixes $\langle e_1\rangle$, and thus $\{I_2,A,B,C\}$ fixes $\langle e_1\rangle$.  
		Since $C$ is arbitrary, we conclude that $\mathcal{F}$ fixes $\langle e_1\rangle$ and this completes the proof.
	\end{proof}
	Now, we give a proof to Theorem~\ref{thm:HM-GL}.
	\begin{proof}[Proof of Theorem~\ref{thm:HM-GL}]
		Let $\mathcal{F}\subset \gl{2}{q}$ be an intersecting set. By Lemma~\ref{lem:fix-line}, we know that $\mathcal{F}$ is contained in the stabilizer of the line $\langle e_1\rangle$. The stabilizer of $\langle e_1\rangle$ in $\gl{2}{q}$ is the subgroup
		\begin{align*}
			V_{0,0} \cup V_{1,0}\cup \ldots \cup V_{q-2,0}.
		\end{align*}
		Using the fact that the identity element is in $\mathcal{F}$, we know that by Corollary~\ref{cor:one-layer}, the cocliques of the subgraph of $\Gamma_q$ induced by $V_{0,0} \cup V_{1,0}\cup \ldots \cup V_{q-2,0}$ which contains the identity element must be contained in $V_{0,0}$, which is the stabilizer of $e_1$ in $\gl{1}{q}$, or
		\begin{align*}
			\bigcup_{i = 0}^{q-2} \omega^i \mathcal{E}(\omega^{-i})
		\end{align*}
		which is the stabilizer of the line $e_2 +\langle e_1\rangle$. This completes the proof.
	\end{proof}
	
	\section{Future work and concluding remark}
	
	In this paper, we showed that whenever $q$ is an odd prime power, then the largest intersecting sets of $\gammal{2}{q}$ in its action on non-zero vectors of $\mathbb{F}_q^2$ were characterized in Theorem~\ref{thm1}, and we showed that they need not be subgroups, unless when $q = p^2$. When $q$ is a power $2$, the largest intersecting sets of $\gammal{2}{q}$ are more complicated, but we characterized them in Theorem~\ref{thm3} and Theorem~\ref{thm4}.
	
	One of the main motivation to study these groups is due to the fact that $\gammal{2}{q}$ is a split extension of $\gl{2}{q}$ by the Galois group of the field. An interesting direction to explore is to characterize the maximum intersecting sets of permutation groups which are the semidirect product $G\rtimes C_k$, for some $k\geq 2$, provided that enough information is known about the maximum intersecting sets of $G$. This question was asked in \cite{ahmadi2015erdHos} for $k = 2$ by Ahmadi and Meagher. Even for the case of the permutation group $\pgl{2}{q}=\psl{2}{q} \rtimes C_2$, for $q\equiv 1 \pmod 4$, acting on cosets of the dihedral group $\dih{q-1}$, it is still not known whether it has the EKR property, let alone a characterization of the largest intersecting sets.
	
	\vspace*{0.6cm}
	\noindent {\bf Acknowledgement: } The research of both authors is supported in part by the Ministry of Education, Science and Sport of Republic of Slovenia (University of Primorska Developmental
	funding pillar). The second author is also supported by the Slovenian Research Agency through the research project J1-50000.
	

\begin{thebibliography}{10}
		
		\bibitem{ahanjideh2022largest}
		M.~Ahanjideh.
		\newblock On the largest intersecting set in ${GL}_2(q)$ and some of its
		subgroups.
		\newblock {\em C. R. Math. Acad. Sci. Paris}, {\bf 360}(G5):497--502, 2022.
		
		\bibitem{ahmadi2015erdHos}
		B.~Ahmadi and K.~Meagher.
		\newblock The {E}rd{\H{o}}s-{K}o-{R}ado property for some 2-transitive groups.
		\newblock {\em Ann. Comb.}, {\bf 19}(4):621--640, 2015.
		
		\bibitem{behajaina2022intersection}
		A.~Behajaina, R.~Maleki, and A.~S. Razafimahatratra.
		\newblock Intersection density of imprimitive groups of degree $pq$.
		\newblock {\em J. Combin. Ser. A (To appear)}, 2024.
		
		\bibitem{forbidden}
		D.~Ellis, G.~Kindler, and N.~Lifshitz.
		\newblock Forbidden intersection problems for families of linear maps.
		\newblock {\em Discrete Anal.}, {\bf 19}:32pp, 2023.
		
		\bibitem{erdos1961intersection}
		P.~Erd\H{o}s, C.~Ko, and R.~Rado.
		\newblock Intersection theorems for systems of finite sets.
		\newblock {\em Q. J. Math.}, {\bf 12}(1):313--320, 1961.
		
		\bibitem{ernst2023intersection}
		A.~Ernst and K.-U. Schmidt.
		\newblock Intersection theorems for finite general linear groups.
		\newblock In {\em Math. Proc. Cambridge Philos. Soc.}, volume 175, pages
		129--160. Cambridge University Press, 2023.
		
		\bibitem{godsil2016erdos}
		C.~Godsil and K.~Meagher.
		\newblock {\em {E}rd\H{o}s-{K}o-{R}ado {T}heorems: {A}lgebraic {A}pproaches}.
		\newblock Cambridge University Press, 2016.
		
		\bibitem{maleki2021no}
		R.~Maleki and A.~S. Razafimahatratra.
		\newblock No {H}ilton-{M}ilner type results for linear groups of degree two.
		\newblock {\em arXiv preprint arXiv:2111.03829}, 2021.
		
		\bibitem{maleki2023erd}
		R.~Maleki and A.~S. Razafimahatratra.
		\newblock The {E}rd{\H{o}}s-{K}o-{R}ado theorem for non-quasiprimitive groups
		of degree $3 p$.
		\newblock {\em arXiv preprint arXiv:2309.09906}, 2023.
		
		\bibitem{meagher2022some}
		K.~Meagher and A.~S. Razafimahatratra.
		\newblock Some {E}rd{\H{o}}s-{K}o-{R}ado results for linear and affine groups
		of degree two.
		\newblock {\em Art Discrete Appl. Math.}, {\bf 6}(1):1, 2022.
		
		\bibitem{meagher2011erdHos}
		K.~Meagher and P.~Spiga.
		\newblock An {E}rd{\H{o}}s--{K}o--{R}ado theorem for the derangement graph of
		${PGL}(2, q)$ acting on the projective line.
		\newblock {\em J. Combin. Ser. A}, {\bf 118}(2):532--544, 2011.
		
		\bibitem{meagher2016erdHos}
		K.~Meagher, P.~Spiga, and P.~H. Tiep.
		\newblock An {E}rd{\H{o}}s--{K}o--{R}ado theorem for finite 2-transitive
		groups.
		\newblock {\em European J. Combin.}, {\bf 55}:100--118, 2016.
		
		\bibitem{spiga2019erdHos}
		P.~Spiga.
		\newblock The {E}rd{\H{o}}s-{K}o-{R}ado theorem for the derangement graph of
		the projective general linear group acting on the projective space.
		\newblock {\em J. Combin. Ser. A}, {\bf 166}:59--90, 2019.
		
	\end{thebibliography}
\end{document}